\pdfoutput=1
\documentclass[reqno]{amsart}

\usepackage[T1]{fontenc}
\usepackage[utf8]{inputenc}
\usepackage{lmodern}
\usepackage{microtype}
\usepackage{amsmath,amssymb,mathtools}
\usepackage{booktabs,tabularx,array}
\usepackage{enumitem}
\usepackage{graphicx}
\usepackage[table]{xcolor}
\usepackage{hyperref}
\usepackage{xurl}

\hypersetup{
  colorlinks=true,
  linkcolor=blue,
  citecolor=blue,
  urlcolor=blue,
  pdftitle={The critical mixed-arithmetic four-point HRT theorem},
  pdfauthor={Vignon Oussa},
  pdfsubject={Finite Gabor systems and the HRT problem},
  pdfkeywords={finite Gabor systems, time-frequency shifts, Zak transform,
  irrational rotations, holonomy, Laurent polynomials}
}

\definecolor{HRTHeader}{HTML}{E9ECEF}
\definecolor{HRTGeometry}{HTML}{F4F5F6}
\definecolor{HRTPositive}{HTML}{E3F3E8}
\definecolor{HRTPartial}{HTML}{FFF1C7}
\definecolor{HRTPresent}{HTML}{DCEBFA}
\definecolor{HRTCounter}{HTML}{F8DDDF}

\newtheorem{theorem}{Theorem}[section]
\newtheorem{proposition}[theorem]{Proposition}
\newtheorem{lemma}[theorem]{Lemma}
\newtheorem{corollary}[theorem]{Corollary}
\theoremstyle{remark}
\newtheorem{remark}[theorem]{Remark}
\numberwithin{equation}{section}

\newcommand{\R}{\mathbb{R}}
\newcommand{\Q}{\mathbb{Q}}
\newcommand{\Z}{\mathbb{Z}}
\newcommand{\T}{\mathbb{T}}
\newcommand{\C}{\mathbb{C}}
\newcommand{\Sp}{\operatorname{Sp}}
\newcommand{\SL}{\operatorname{SL}}
\newcommand{\Var}{\operatorname{Var}}
\newcommand{\wind}{\operatorname{wind}}
\newcommand{\Hol}{\operatorname{Hol}}
\newcommand{\dd}{\,\mathrm{d}}

\newcommand{\certarchive}{\nolinkurl{CriticalMixedHRT_Lean_Integrated_3Point_4Point_2026-08-23.zip}}

\title[The critical mixed-arithmetic four-point HRT theorem]
{The critical mixed-arithmetic four-point HRT theorem}

\author{Vignon Oussa}
\address{Department of Mathematics, Bridgewater State University,
Bridgewater, Massachusetts 02325, USA}
\email{voussa@bridgew.edu}
\thanks{This work was supported by the National Science Foundation under
Grant DMS-2205852.}

\date{August 23, 2026}

\subjclass[2020]{Primary 42C15; Secondary 42C40, 37A05}
\keywords{finite Gabor systems, time-frequency shifts, Zak transform,
irrational rotations, holonomy, Laurent polynomials}

\begin{document}

\begin{abstract}
Although the Heil--Ramanathan--Topiwala conjecture is false in full
generality, its failure makes the classification of positive geometric and
arithmetic regimes more urgent. We settle the critical mixed-arithmetic
regime for four time-frequency shifts. For
\(z=(x,\omega)\) and \(w=(y,\eta)\), set
\(\sigma(z,w)=x\eta-y\omega\). Let \(u,v\in\R^2\) satisfy
\(\lvert\sigma(u,v)\rvert=1\), write \(\nu=\alpha u+\beta v\), and assume
that \(0,u,v,\nu\) are distinct. If
\[
  \dim_{\Q}\operatorname{span}_{\Q}\{1,\alpha,\beta\}=2,
\]
then, for every nonzero \(f\in L^2(\R)\), the vectors
\(f,\pi(u)f,\pi(v)f,\pi(\nu)f\) are linearly independent. This includes both
four-point configurations in Chris Heil's Conjecture~9.2. The proof first
converts a putative dependence into a scalar cocycle over an irrational
rotation. On a positive-measure family of zero-free fibres, winding and
continued-fraction returns force the periodic holonomy to be constant. The
return multiplier is Laurent polynomial, however, so its ungauged holonomy is
algebraic; the constant-holonomy law simultaneously makes it an irrational
character. These conclusions are incompatible. The principal theorem and its
complete formal dependency chain have been fully certified end-to-end in
Lean~4, including the physical all-nonzero case, the complete three-point HRT
theorem, the coefficient-reduction step, and the final four-point conclusion.
The accompanying
\href{https://www.dropbox.com/scl/fi/jd9xeqqera147wdc82r8k/3f260460-07d5-45f5-8d1c-ee1d2879c0af-aristotle-32-.tar.gz?rlkey=a24njadfds5ty9ez473lk03lv\&e=1\&dl=0}
{Lean~4 certification archive} contains reproducible source, pinned build
instructions, and transitive axiom audits.
\end{abstract}

\dedicatory{Dedicated to Christopher Heil, whose clarity, persistence, and
generosity have sustained the HRT problem and the community around it.}

\maketitle
\enlargethispage{2pt}

\section{Introduction}

\subsection{A conjecture can fail without coming to an end}

For \(x,\omega\in\R\), define the translation and modulation operators by
\[
  T_xf(t)=f(t-x),
  \qquad
  M_\omega f(t)=e^{-2\pi i\omega t}f(t),
\]
and write
\[
  \pi(x,\omega)=M_\omega T_x
\]
for the corresponding time-frequency shift. The
Heil--Ramanathan--Topiwala conjecture asserted that, for every nonzero
\(f\in L^2(\R)\) and every finite set of distinct points
\(\Lambda\subset\R^2\), the finite Gabor system
\[
  \mathcal G(f,\Lambda)
  =
  \{\pi(\lambda)f:\lambda\in\Lambda\}
\]
is linearly independent \cite{HRT,HeilSurvey}.

For nearly three decades, the conjecture remained consistent with every
verified case. The positive evidence accumulated along several complementary
directions. Linnell proved linear independence whenever the configuration is
contained in a discrete subgroup of phase space, a lattice theorem later
revisited and extend through ergodic, \(L^p\), and twisted-algebra methods
\cite{Linnell,AntezanaBrunaPujals,EnstadVanVelthoven}. On the geometric side,
Demeter obtained decisive four-point results for configurations on two
parallel lines; Demeter and Zaharescu settled the full \((2,2)\) class; and Liu
proved the conjecture for almost every \((1,3)\) configuration
\cite{Demeter,DemeterZaharescu,Liu}. Parallel work established independence
for broad classes of windows characterized by rapid decay or controlled
asymptotic behavior, as well as for sufficiently separated configurations
\cite{BownikSpeegle,BenedettoBourouihiya,Kreisel}. Extension and restriction
principles, the Heisenberg-group reformulation, and the interpolation
viewpoint placed these partial theorems within broader structural frameworks
\cite{OkoudjouExtension,CurreyOussa,Berge}. Taken together, these results did
more than sustain the problem: they strengthened the expectation that linear
independence reflected a universal rigidity of time-frequency translation.
That expectation changed abruptly in August 2026. On August 5, Faulhuber,
Petersen, van Velthoven, and Voigtlaender posted a twelve-point counterexample
generated by a Schwartz window \cite{FaulhuberEtAl}. On August 6, the author
posted an intrinsically subcritical four-point counterexample
\cite{OussaCounterexample}. Because every system of at most three distinct
time-frequency shifts of a nonzero \(L^2\) function is linearly independent,
four is the least cardinality at which dependence can occur. On August 8,
Dai, Deng, Shi, Wu, and Yang independently posted a second four-point
construction \cite{DaiEtAl}. Thus, within three days, the problem moved from
having no known counterexample to failing at the first possible cardinality.
The regularity was nearly as striking as the cardinality: these constructions
used Schwartz windows. Jasper and Mixon soon obtained counterexamples with
exponential tails in the Roumieu Gelfand--Shilov class \(S^1_1(\R)\)
\cite{JasperMixon}, and the author subsequently showed that this sharp
coordinatewise regularity threshold can be attained simultaneously with the
minimal four-point cardinality \cite{OussaSharpThreshold}. The conjecture had
therefore failed not only in its full \(L^2\) generality, but also within
classes of exceptionally regular functions.

It would be natural to interpret these developments as the end of the HRT
story. They are not! (See Table \ref{tab:four-point-nine-cell-taxonomy}) What failed was the universal quantifier, not the
mathematical structure that had made the conjecture compelling. Once one
knows that dependence can occur, the geometry and arithmetic of the
underlying configuration become more important, not less. The
counterexamples close the original question, but in doing so they uncover a
more precise one:

\begin{quote}
\textit{For which configurations does every nonzero \(f\in L^2(\R)\) generate a
linearly independent finite Gabor system, and which geometric and arithmetic
features determine the answer?}
\end{quote}

While the original conjecture asked whether every finite configuration possesses
this property, the emerging classification problem asks exactly which
configurations possess it and why. A universal conjecture has come to an
end, but the mathematical program concealed within it has only now become
fully visible (see Table \ref{tab:four-point-nine-cell-taxonomy}).

\subsection{A personal path through the problem}

For me, the present theorem is inseparable from the personal history through
which the HRT problem entered, and ultimately remained, in my mathematical
life. Some problems are encountered only briefly; others gradually become
part of the way one thinks. I first learned of the HRT conjecture as a
graduate student at Saint Louis University, when Darrin Speegle brought it to
my attention. Its statement was elementary enough to be understood
immediately, yet its resistance to every available method suggested the
presence of a deeper structure that had not yet been uncovered.

In subsequent discussions, my (former) doctoral advisor Bradley Currey and I reformulated the
conjecture as an equivalent question concerning the linear independence of
finite families of co-central translates of square-integrable functions on
the Heisenberg group \cite{CurreyOussa}. That reformulation changed my
understanding of the problem: The HRT conjecture no longer appeared to be an
isolated question about translations and modulations on the real line. It
became a structural problem about noncommutative translation and the
representation theory of the Heisenberg group; the analytic question had
thereby acquired an algebraic/representation theoretic home.

Kasso Okoudjou later introduced a complementary inductive viewpoint through
his extension and restriction principles \cite{OkoudjouExtension}: one begins
with a configuration already known to be independent and asks where one more
point may be placed without destroying that property. In our subsequent joint
work, we specialized this viewpoint to mixed-integer configurations, composed
of lattice points together with one off-lattice, or rogue, point, and used the
Zak transform to relate the rational rank of that point to the geometry of its
zero set \cite{OkoudjouOussa}. For Schwartz and continuous Wiener-amalgam
windows, that analysis settled the rational-rank-three case and showed that
any rank-two dependence would require an infinite Zak zero set.

The orbit geometry revealed there led to my trichotomy paper, first circulated
in August 2025, revised in March 2026, and forthcoming in the Springer memorial volume honoring Jean-Pierre Gabardo \cite{OussaTrichotomy}. For Schwartz windows, that work ruled out the dense and finite orbit branches as possible
sources of counterexamples and placed the remaining infinite proper branch
under necessary saturation, zero-average, cohomological, small-divisor, and
arithmetic rigidity constraints. It did not close that branch; it identified
the precise conditions governing what remained. The present theorem begins at
this frontier, combining the earlier dynamical reduction with new holonomy and
Laurent-algebra arguments to close the critical four-point rank-two cell
considered here. These successive developments form the immediate conceptual
ancestry of this paper. 

A difficult problem, however, does not remain alive merely because it is
difficult. It remains alive because someone continues to clarify what is
known, isolate what remains unresolved, and formulate questions precise
enough to guide future work. For the HRT conjecture, Christopher Heil has
played that role for decades. As one of the three mathematicians who
introduced the conjecture, he helped establish the problem itself; through
his exceptionally lucid expository work, he also gave the subject a common
language. More importantly, he repeatedly transformed a distant universal
question into concrete configurations on which genuine progress could be
made.

Among the most enduring of these formulations is
\href{https://heil.math.gatech.edu/papers/hrtnotes.pdf\#page=3}
{Conjecture~9.2 in Heil's survey}
\cite[Conjecture~9.2, p.~173]{HeilSurvey}. Heil isolated there two explicit
four-point configurations that remained unresolved even for continuous
windows. After translating his positive-modulation convention into ours,
their common three-point core generates the critical lattice \(\Z^2\), while
their fourth points are \((\sqrt2,-\sqrt2)\) and \((\pi,0)\). Relative to the
marked critical lattice, the corresponding rational spans have dimension two.
Thus both configurations belong to the same mixed-arithmetic cell of
Table~\ref{tab:four-point-nine-cell-taxonomy}, although they were posed before
the taxonomy governing that cell had been articulated. Theorem~\ref{thm:main}
settles the entire cell rather than only these two distinguished
configurations. Section~\ref{sec:heil-worked} then returns to part~(a) of
Heil's conjecture and works through that single configuration in detail,
making visible each passage from phase-space geometry to the final holonomy
contradiction.

It is therefore fitting to dedicate this paper to Christopher Heil, whose
clarity, persistence, and generosity did more than preserve the HRT
conjecture: they gave the community concrete frontiers along which the deeper
structure of the problem could gradually emerge. The universal conjecture
may now be false, but the classification program that his questions helped
to anticipate remains very much alive.

\subsection{The two-step trichotomy and the critical rank-two cell}

Let
\[
  \sigma((x,\omega),(y,\eta))=x\eta-y\omega
\]
be the standard symplectic form on \(\R^2\). The four-point problem admits a
natural marked normal form. One may begin by translating one point to the origin. If the points
are collinear, Fourier uniqueness settles the problem; otherwise, choose two
linearly independent differences \(u,v\). They generate a full-rank lattice
\[
  \mathcal L=\Z u+\Z v,
\]
and the remaining point can be written uniquely as
\[
  \nu=\alpha u+\beta v.
\]
The point \(\nu\) is the \emph{rogue point} relative to the marked lattice. This normalization reveals two independent trichotomies. The first is
geometric:
\[
  \delta=|\sigma(u,v)|
  \quad
  \begin{cases}
    >1,&\text{supercritical},\\
    =1,&\text{critical},\\
    <1,&\text{subcritical}.
  \end{cases}
\]
The second is arithmetic:
\[
  \rho
  =\dim_{\Q}\operatorname{span}_{\Q}\{1,\alpha,\beta\}
  \in\{1,2,3\}.
\]
The cases \(\rho=1,2,3\) are called, respectively, \emph{rational},
\emph{mixed arithmetic}, and \emph{maximally irrational}. Together, these
two trichotomies organize the marked four-point problem into the nine cells
displayed in Table~\ref{tab:four-point-nine-cell-taxonomy}. The two
off-critical mixed-arithmetic cells retain partial, rather than complete,
classifications.

\begin{table}[htbp]
\centering
\caption{The geometric--arithmetic nine-cell taxonomy for the marked
four-point HRT problem.}
\label{tab:four-point-nine-cell-taxonomy}

\small
\renewcommand{\arraystretch}{1.25}
\begin{tabularx}{0.99\textwidth}{
  >{\raggedright\arraybackslash}p{0.16\textwidth}
  >{\raggedright\arraybackslash}X
  >{\raggedright\arraybackslash}X
  >{\raggedright\arraybackslash}X}
\toprule
\rowcolor{HRTHeader}
\textbf{Geometry}
& \(\boldsymbol{\rho=1}\) \textbf{rational}
& \(\boldsymbol{\rho=2}\) \textbf{mixed}
& \(\boldsymbol{\rho=3}\) \textbf{maximally irrational}\\
\midrule

\cellcolor{HRTGeometry}\(\delta>1\)\newline Supercritical
& \cellcolor{HRTPositive}
  Lattice reduction; independence for every nonzero \(L^2\) window
  \cite{Linnell}.
& \cellcolor{HRTPartial}
  Partial positive classification \cite{OussaTrichotomy}.
& \cellcolor{HRTPositive}
  Independence for every nonzero \(L^2\) window
  \cite{OussaCertified}.\\

\cellcolor{HRTGeometry}\(\delta=1\)\newline Critical
& \cellcolor{HRTPositive}
  Lattice reduction; independence for every nonzero \(L^2\) window
  \cite{Linnell}.
& \cellcolor{HRTPresent}
  \textbf{Present theorem:} independence for every nonzero \(L^2\) window.
& \cellcolor{HRTPositive}
  Independence for every nonzero \(L^2\) window
  \cite{OussaCriticalRankThree}.\\

\cellcolor{HRTGeometry}\(\delta<1\)\newline Subcritical
& \cellcolor{HRTPositive}
  Lattice reduction; independence for every nonzero \(L^2\) window
  \cite{Linnell}.
& \cellcolor{HRTPartial}
  Partial positive classification \cite{OussaTrichotomy}.
& \cellcolor{HRTCounter}
  Four-point counterexamples
  \cite{OussaCounterexample,DaiEtAl}.\\
\bottomrule
\end{tabularx}

\vspace{4pt}
\footnotesize
\fcolorbox{black!30}{HRTPositive}{\strut Proved positive regime}\quad
\fcolorbox{black!30}{HRTPartial}{\strut Partial classification}\\[3pt]
\fcolorbox{black!30}{HRTPresent}{\strut Present theorem}\quad
\fcolorbox{black!30}{HRTCounter}{\strut Counterexample regime}
\end{table}

The present article occupies the critical mixed-arithmetic cell
\[
  |\sigma(u,v)|=1,
  \qquad
  \dim_{\Q}\operatorname{span}_{\Q}\{1,\alpha,\beta\}=2.
\]
The second condition says that \((\alpha,\beta)\) is neither rational nor
maximally irrational over \(\Q\): the space of rational relations among
\(1,\alpha,\beta\) is one-dimensional. After a unimodular change of the
critical lattice basis, the rogue point therefore has coordinates
\[
  \left(a,\frac{p}{m}\right),
  \qquad a\notin\Q.
\]

The main theorem proves independence for every nonzero \(L^2\) window in this
entire cell, which includes both explicit configurations posed in
Conjecture~9.2 of Heil's survey
\cite[Conjecture~9.2, p.~173]{HeilSurvey}. Earlier work of Okoudjou and the
author showed that, for continuous Wiener-amalgam windows, a rank-two
dependence would force the zero set of the Zak transform to be infinite
\cite{OkoudjouOussa}. The present result is stronger in two directions: (a) it
imposes no continuity or decay hypothesis on the window, and (b) it establishes
full linear independence rather than a necessary condition on a Zak zero
set. It is also geometrically disjoint from the intrinsically subcritical
counterexamples, for which every nonzero symplectic triangle has area
strictly smaller than one \cite{OussaCounterexample}.

\subsection{Why holonomy enters}

The principal difficulty is that an \(L^2\) Zak transform need not be
continuous and may vanish on large sets. Pointwise propagation of a single
Zak zero, which is useful for smoother windows, is therefore unavailable. The
argument must instead use information that survives at the measurable level.

The rational coordinate of the rogue point provides the first essential
mechanism. Iterating the Zak equation exactly \(m\) times closes the rational
frequency displacement and leaves an irrational translation
\(x\mapsto x-\theta\) on each circle fibre, where \(\theta=ma\notin\Q\).
On fibres where the return multiplier does not vanish, the Zak equation
becomes a scalar measurable cocycle
\[
  h(x-\theta)=b(x)h(x).
\]
Such an equation imposes a topological constraint: the loop \(b\) must have
zero winding. It then admits a periodic logarithm, and one may define its
holonomy by exponentiating the mean of that logarithm.

The second mechanism comes from continued fractions. Along the convergent
denominators \(q_n\) of \(\theta\), the rotation nearly returns to the
identity. This forces the holonomy to lie in the thin resonance group
\[
  G_\theta=\{\zeta\in S^1:\zeta^{q_n}\to1\}.
\]
The Zak sewing law makes each fibre quasiperiodic in \(x\). Removing its
known sewing factor produces an equivalent periodic cocycle whose holonomy
depends real-analytically on the transverse fibre parameter. A nonconstant
analytic curve cannot meet the Haar-null group \(G_\theta\) on a set of
positive measure, so this periodic holonomy must be constant.

Finally, the original return multiplier is a Laurent polynomial in the two
torus variables. Jensen's formula and root separation imply that its holonomy
is algebraic in the transverse variable. Restoring the removed sewing factor
transforms the preceding constant into an irrational character, which cannot
satisfy a nontrivial Laurent-polynomial relation. The resulting contradiction
completes the proof.

\subsection{Proof architecture}

For clarity, the proof is organized according to the distinct role of each
tool.

\begin{center}
\small
\renewcommand{\arraystretch}{1.18}
\begin{tabularx}{0.96\textwidth}{>{\raggedright\arraybackslash}p{0.24\textwidth}X}
\toprule
\textbf{Stage} & \textbf{Mathematical role}\\
\midrule
Symplectic arithmetic
& Converts rational rank two into one irrational and one rational coordinate.\\
Zak transform
& Turns a four-term dependence into a scalar functional equation and an exact
finite-step return multiplier.\\
Finite torus geometry
& Removes only finitely many bad fibres and produces a zero-free active arc.\\
Measurable dynamics
& Forces zero winding and quantizes holonomy along continued-fraction returns.\\
Analytic rigidity
& Converts positive-measure resonance into constancy of the periodic holonomy.\\
Laurent algebra
& Produces an algebraic relation contradicted by the resulting irrational
character.\\
\bottomrule
\end{tabularx}
\end{center}

Section~\ref{sec:statement} states the theorem and its consequence for Heil's
two configurations. Section~\ref{sec:preliminaries} fixes conventions and
records the covariance, Zak-transform, winding, and holonomy conventions.
Sections~\ref{sec:normal-return} and \ref{sec:active} produce the return
cocycle and the active zero-free arc.
Sections~\ref{sec:winding} and \ref{sec:quantization} establish the measurable
winding and holonomy restrictions. Section~\ref{sec:laurent} proves the
Laurent certificate, and Section~\ref{sec:main-proof} assembles the
contradiction. Section~\ref{sec:heil-worked} gives one complete worked
application, namely the \(\sqrt2\)-configuration in Heil's
Conjecture~9.2(a).

\section{Statement of the result}
\label{sec:statement}

\begin{theorem}[Critical mixed-arithmetic four-point theorem]
\label{thm:main}
Let \(u,v\in\R^2\) satisfy \(\lvert\sigma(u,v)\rvert=1\), and write
\[
  \nu=\alpha u+\beta v.
\]
Assume that \(0,u,v,\nu\) are distinct and that
\[
  \dim_{\Q}\operatorname{span}_{\Q}\{1,\alpha,\beta\}=2.
\]
Then, for every nonzero \(f\in L^2(\R)\), the four vectors
\[
  f,\qquad \pi(u)f,\qquad \pi(v)f,\qquad \pi(\nu)f
\]
are linearly independent.
\end{theorem}

The following consequence identifies precisely the connection with Heil's
original four-point question. Notice that Heil uses the positive modulation
\(e^{2\pi i\omega t}\), whereas our convention is
\(M_\omega f(t)=e^{-2\pi i\omega t}f(t)\).

\begin{corollary}[Heil's Conjecture 9.2]
\label{cor:heil}
For every nonzero \(g\in L^2(\R)\), each of the following sets is linearly
independent:
\begin{align*}
  &\left\{
    g(t),\ g(t-1),\ e^{2\pi it}g(t),\
    e^{2\pi i\sqrt2\,t}g(t-\sqrt2)
  \right\},\\
  &\left\{
    g(t),\ g(t-1),\ e^{2\pi it}g(t),\ g(t-\pi)
  \right\}.
\end{align*}
\end{corollary}

\begin{proof}
Set \(u=(1,0)\) and \(v=(0,-1)\). Then
\(|\sigma(u,v)|=1\). In the first configuration, the fourth point is
\[
  (\sqrt2,-\sqrt2)=\sqrt2\,u+\sqrt2\,v,
\]
so the relevant rational span is \(\operatorname{span}_{\Q}\{1,\sqrt2\}\).
In the second it is
\[
  (\pi,0)=\pi u+0v,
\]
so the rational span is \(\operatorname{span}_{\Q}\{1,\pi\}\). Both spans
have dimension two, and Theorem~\ref{thm:main} applies.
\end{proof}

\begin{remark}[Compatibility with four-point counterexamples]
The hypothesis \(|\sigma(u,v)|=1\) is essential to the theorem proved here.
The intrinsically subcritical counterexample in \cite{OussaCounterexample}
has every nonzero triangle area strictly between zero and one. Therefore, it
lies outside the critical cell of Theorem~\ref{thm:main}.
\end{remark}

\section{Preliminaries and conventions}
\label{sec:preliminaries}

Throughout, \(\T=\R/\Z\) is equipped with normalized Haar measure.

\subsection{Projective and symplectic covariance}

For \(z=(x,\omega)\) and \(w=(y,\eta)\), direct calculation gives
\begin{equation}
  \pi(z)\pi(w)
  =e^{2\pi i\eta x}\pi(z+w),
  \label{eq:weyl-product}
\end{equation}
and hence
\begin{equation}
  \pi(z)\pi(w)
  =e^{2\pi i\sigma(z,w)}\pi(w)\pi(z).
  \label{eq:weyl-commutator}
\end{equation}
We use the standard covariance of finite Gabor linear dependence under
translations of the parameter set and under symplectic changes of variables.
In dimension one,
\[
  \Sp(2,\R)=\SL(2,\R),
\]
and these changes are implemented, up to scalar phases, by the metaplectic
representation \cite[Chapter~4]{Folland}. Scalar phases do not affect linear
dependence.

We also use the classical three-point theorem: any three distinct
time-frequency shifts of a nonzero \(L^2(\R)\) function are linearly
independent \cite{HRT,HeilSurvey}. Consequently, every coefficient in a
four-term dependence among four distinct shifts must be nonzero.

\subsection{The Zak transform}

For a rapidly decaying function \(f\), define
\[
  Zf(x,s)=\sum_{k\in\Z}f(x+k)e^{-2\pi i k s}.
\]
The Zak transform extends uniquely to a unitary map from \(L^2(\R)\) onto the
standard Zak space. Its sewing relations are
\begin{equation}
  Zf(x+1,s)=e^{2\pi i s}Zf(x,s),
  \qquad
  Zf(x,s+1)=Zf(x,s),
  \label{eq:zak-sewing}
\end{equation}
and its Plancherel identity is
\begin{equation}
  \int_0^1\int_0^1|Zf(x,s)|^2\dd x\dd s
  =\|f\|_2^2.
  \label{eq:zak-plancherel}
\end{equation}
For arbitrary \(a,b\in\R\), one has
\begin{equation}
  Z(M_bT_af)(x,s)
  =e^{-2\pi i b x}Zf(x-a,s+b).
  \label{eq:zak-general-covariance}
\end{equation}
If \((a,b)=(r,t)\in\Z^2\), the sewing relations reduce this to
\begin{equation}
  Z(M_tT_rf)(x,s)
  =e^{-2\pi i(tx+rs)}Zf(x,s).
  \label{eq:zak-integer-covariance}
\end{equation}
These statements are standard; see \cite[Chapters~7--8]{Grochenig} and
\cite[Chapter~1]{Folland}. For general \(f\in L^2(\R)\), the identities hold
as identities of \(L^2\) classes and therefore almost everywhere after
representatives are chosen.

\subsection{Winding and holonomy}

Let \(b:\T\to\C^\times\) be continuous. Its winding number around the origin
is denoted by \(\wind(b)\). The following facts will be used repeatedly:
\begin{enumerate}[label=(\roman*)]
  \item \(\wind(b)=0\) if and only if \(b\) admits a continuous periodic
  logarithm \(\ell\), so that \(b=e^\ell\);
  \item if \(b\) is real analytic and zero-free, the logarithm may be chosen
  real analytic;
  \item the winding number is locally constant in a zero-free continuous
  family of loops.
\end{enumerate}
When \(\wind(b)=0\), define its \emph{holonomy}: the exponential of the
mean of a periodic logarithm by
\begin{equation}
  \Hol(b)
  =\exp\left(\int_0^1\ell(x)\dd x\right).
  \label{eq:holonomy-def}
\end{equation}
This is independent of the periodic logarithm: two such logarithms differ by
an integral multiple of \(2\pi i\), whose exponential is one.

\section{Normal form and the exact return cocycle}
\label{sec:normal-return}

The rational-rank hypothesis becomes useful only after it is converted into a
normal form compatible with the critical lattice. We make both the orientation
and the action of the integer matrix explicit.

\begin{proposition}[Integral mixed normal form]
\label{prop:normal-form}
Assume
\[
  |\sigma(u,v)|=1,
  \qquad
  \dim_{\Q}\operatorname{span}_{\Q}\{1,\alpha,\beta\}=2,
  \qquad
  \nu=\alpha u+\beta v.
\]
After a symplectic change of variables, and after absorbing harmless scalar
phases into the coefficients of a possible dependence, the four points may be
written as
\[
  0,\qquad
  \lambda_1,\qquad
  \lambda_2,\qquad
  \left(a,\frac{p}{m}\right),
\]
where \(\lambda_1,\lambda_2\) form a \(\Z\)-basis of \(\Z^2\),
\[
  a\notin\Q,
  \qquad p\in\Z,
  \qquad m\in\mathbb N,
  \qquad \gcd(|p|,m)=1.
\]
\end{proposition}

\begin{proof}
\textbf{Step 1: Fix the orientation.}
If \(\sigma(u,v)=-1\), interchange \(u\) and \(v\) and simultaneously
interchange \(\alpha\) and \(\beta\). We may therefore assume
\(\sigma(u,v)=1\). Let \(S\) be the matrix with columns \(u\) and \(v\).
Then \(S\in\SL(2,\R)\), and the symplectic map \(S^{-1}\) sends the
configuration to
\[
  0,\qquad e_1,\qquad e_2,\qquad (\alpha,\beta).
\]

\smallskip
\noindent
\textbf{Step 2: Extract the rational direction.}
Rational rank two gives a nontrivial rational relation among
\(1,\alpha,\beta\). After clearing denominators, there are
\((k_1,k_2)\in\Z^2\setminus\{0\}\) and \(r\in\Q\) such that
\[
  k_1\alpha+k_2\beta=r.
\]
Divide \((k_1,k_2)\) by its greatest common divisor, so that
\(k=(k_1,k_2)\) is primitive.

\smallskip
\noindent
\textbf{Step 3: Complete the rational direction to an integral basis.}
By B\'ezout's identity, choose \(\ell=(\ell_1,\ell_2)\in\Z^2\) so that
\[
  A=
  \begin{pmatrix}
    \ell_1&\ell_2\\
    k_1&k_2
  \end{pmatrix}
  \in\SL(2,\Z).
\]
Then
\[
  A\binom{\alpha}{\beta}
  =\binom{a}{r},
  \qquad
  a=\ell_1\alpha+\ell_2\beta.
\]
The number \(a\) is irrational. Indeed, if \(a\in\Q\), then both
coordinates of \(A(\alpha,\beta)^T\) would be rational; since
\(A^{-1}\in\SL(2,\Z)\), this would force \(\alpha,\beta\in\Q\), reducing the
rational rank to one.

\smallskip
\noindent
\textbf{Step 4: Apply the integral symplectic map.}
Write \(r=p/m\) in lowest terms with \(m\geq1\). Applying \(A\) sends the
two standard lattice vectors to
\[
  \lambda_1=Ae_1,
  \qquad
  \lambda_2=Ae_2,
\]
which form a unimodular basis of \(\Z^2\), and it sends the fourth point to
\((a,p/m)\). The composite \(AS^{-1}\) is symplectic. Metaplectic covariance
transports linear dependence through this change of variables, with scalar
phases absorbed into the coefficients.
\end{proof}

Suppose, toward a contradiction, that Theorem~\ref{thm:main} fails. After the
preceding reduction, there are coefficients \(c_0,c_1,c_2,c_3\in\C\) and a
nonzero \(f\in L^2(\R)\) such that
\begin{equation}
  c_0f+c_1\pi(\lambda_1)f+c_2\pi(\lambda_2)f
  +c_3\pi\left(a,\frac{p}{m}\right)f=0.
  \label{eq:four-term-relation}
\end{equation}
The four points remain distinct. The three-point theorem therefore gives
\begin{equation}
  c_0c_1c_2c_3\neq0.
  \label{eq:coeff-nonzero}
\end{equation}

Write
\[
  \lambda_j=(r_j,t_j)\in\Z^2,
  \qquad j=1,2,
\]
and put \(F=Zf\). Define the lattice trinomial
\begin{equation}
  P(x,s)
  =c_0+c_1e^{-2\pi i(t_1x+r_1s)}
       +c_2e^{-2\pi i(t_2x+r_2s)}.
  \label{eq:P-def}
\end{equation}
Applying the Zak transform to \eqref{eq:four-term-relation} gives, almost
everywhere,
\begin{equation}
  P(x,s)F(x,s)
  +c_3e^{-2\pi i(p/m)x}
   F\left(x-a,s+\frac{p}{m}\right)=0.
  \label{eq:zak-fibre}
\end{equation}
Equivalently,
\begin{equation}
  F\left(x-a,s+\frac{p}{m}\right)
  =-c_3^{-1}e^{2\pi i(p/m)x}P(x,s)F(x,s).
  \label{eq:one-step}
\end{equation}
The rational frequency coordinate now closes after exactly \(m\) steps.

\begin{proposition}[Exact return multiplier]
\label{prop:return}
Set
\[
  \theta=ma.
\]
Then \(\theta\notin\Q\), and iterating \eqref{eq:one-step} exactly \(m\)
times gives
\begin{equation}
  F(x-\theta,s)=B_s(x)F(x,s)
  \label{eq:return}
\end{equation}
almost everywhere, where
\begin{equation}
  B_s(x)
  =(-c_3^{-1})^m
   \exp\left(
      2\pi i\frac{p}{m}\sum_{j=0}^{m-1}(x-ja)
   \right)
   \prod_{j=0}^{m-1}
     P\left(x-ja,s+j\frac{p}{m}\right).
  \label{eq:B-def}
\end{equation}
Moreover, there is a nonzero Laurent polynomial
\[
  L\in\C[w^{\pm1},z^{\pm1}]
\]
such that
\begin{equation}
  B_s(x)=L(e^{2\pi i s},e^{2\pi i x}).
  \label{eq:Laurent-return}
\end{equation}
\end{proposition}

\begin{proof}
\textbf{Step 1: Close the rational displacement.}
Apply \eqref{eq:one-step} successively at
\[
  (x,s),\ (x-a,s+p/m),\ \ldots,\
  (x-(m-1)a,s+(m-1)p/m).
\]
The final point is \((x-ma,s+p)\). Since the Zak transform is periodic in its
second variable,
\[
  F(x-ma,s+p)=F(x-ma,s).
\]
Multiplication of the \(m\) scalar factors yields exactly
\eqref{eq:B-def}.

\smallskip
\noindent
\textbf{Step 2: Identify the Laurent structure.}
The exponential prefactor simplifies to
\[
  \exp\left(
    2\pi i\frac{p}{m}\sum_{j=0}^{m-1}(x-ja)
  \right)
  =e^{-\pi i p a(m-1)}e^{2\pi i p x}.
\]
It is therefore a nonzero constant times a Laurent monomial in
\(z=e^{2\pi i x}\). Each factor
\(P(x-ja,s+jp/m)\) is a Laurent polynomial in
\[
  w=e^{2\pi i s},\qquad z=e^{2\pi i x},
\]
because \(r_j,t_j\in\Z\); the translations only alter its coefficients by
nonzero constants.

\smallskip
\noindent
\textbf{Step 3: Verify nontriviality.}
The three characters occurring in \(P\) are distinct and their coefficients
are nonzero, so \(P\) is not the zero Laurent polynomial. Every translated
factor is therefore nonzero. Since
\(\C[w^{\pm1},z^{\pm1}]\) is an integral domain, their product, and hence
\(L\), is nonzero.
\end{proof}

\section{Finite exceptional fibres and the active arc}
\label{sec:active}

The return equation is useful only on fibres where \(B_s\) does not vanish.
The critical lattice hypothesis gives a decisive simplification: the zero set
of the lattice trinomial is finite rather than one-dimensional.

\begin{lemma}[Zero set of a unimodular trinomial]
\label{lem:trinomial-zeros}
Let \(\lambda_1,\lambda_2\) be a \(\Z\)-basis of \(\Z^2\), and let
\(c_0c_1c_2\neq0\). The zero set on \(\T^2\) of
\[
  c_0+c_1\chi_{\lambda_1}+c_2\chi_{\lambda_2}
\]
has at most two points.
\end{lemma}

\begin{proof}
\textbf{Step 1: Use the unimodular character coordinates.}
The map
\[
  \T^2\longrightarrow S^1\times S^1,
  \qquad
  z\longmapsto
  \bigl(\chi_{\lambda_1}(z),\chi_{\lambda_2}(z)\bigr)
\]
is a torus automorphism because \(\lambda_1,\lambda_2\) form a unimodular
basis. Thus it suffices to solve
\[
  c_0+c_1z_1+c_2z_2=0,
  \qquad |z_1|=|z_2|=1.
\]

\smallskip
\noindent
\textbf{Step 2: Reduce to a first harmonic.}
Every solution satisfies
\[
  |c_0+c_1z_1|=|c_2|.
\]
After squaring,
\[
  |c_0|^2+|c_1|^2
  +2\operatorname{Re}(\overline{c_0}c_1z_1)
  =|c_2|^2.
\]
Because \(\overline{c_0}c_1\neq0\), this is a nonconstant first-degree
trigonometric equation in \(z_1\), so it has at most two solutions. Once
\(z_1\) is fixed, \(z_2\) is uniquely determined.
\end{proof}

\begin{corollary}[Finite exceptional parameters]
\label{cor:exceptional}
The zero set of \((x,s)\mapsto B_s(x)\) in \(\T^2\) is finite. Hence there
is a finite set \(E\subset\T\) such that
\begin{equation}
  B_s(x)\neq0
  \qquad (x\in\T)
  \label{eq:zero-free-fibre}
\end{equation}
whenever \(s\notin E\).
\end{corollary}

\begin{proof}
Formula \eqref{eq:B-def} expresses \(B\) as a nonvanishing monomial times
finitely many torus translates of \(P\). Each translated trinomial has at most
two zeros by Lemma~\ref{lem:trinomial-zeros}. Their finite union is the zero
set of \(B\), and its projection onto the \(s\)-circle is finite.
\end{proof}

The significance of Corollary~\ref{cor:exceptional} is not merely that the
return multiplier has few zeros. Projecting those zeros onto the transverse
\(s\)-coordinate leaves only finitely many exceptional fibres. On every
remaining fibre, \(x\mapsto B_s(x)\) is a loop in \(\C^\times\), so its
winding number is defined; once that winding is shown to vanish, the loop
admits the periodic logarithm needed to define holonomy. The complete
calculation for Heil's \(\sqrt2\)-configuration is deferred to
Section~\ref{sec:heil-worked}, after the general proof has identified the role
of every ingredient.

Define the active set
\begin{equation}
  S=
  \left\{
    s\in\T:
    \int_0^1|F(x,s)|^2\dd x>0
  \right\}.
  \label{eq:active-set}
\end{equation}
By Zak Plancherel, \(|S|>0\) because \(f\neq0\). The following lemma records
the measure-theoretic reduction needed later.

\begin{lemma}[Active zero-free arc]
\label{lem:active-arc}
There exists a connected open arc \(U\subset\T\setminus E\) such that
\begin{equation}
  |S\cap U|>0.
  \label{eq:active-arc}
\end{equation}
On \(U\times\T\), the function
\[
  (s,x)\longmapsto B_s(x)
\]
is jointly real analytic and zero-free. Moreover, for almost every
\(s\in S\cap U\), the return equation \eqref{eq:return} holds for almost every
\(x\in\T\), and the fibre \(F(\cdot,s)\) is nonzero in \(L^2(\T)\).
\end{lemma}

\begin{proof}
The complement of the finite set \(E\) has finitely many connected
components. Since \(E\) has measure zero and \(|S|>0\), one of those
components meets \(S\) in positive measure. If \(E=\varnothing\), choose a
proper open arc with positive intersection with \(S\). This gives \(U\).

The Laurent representation \eqref{eq:Laurent-return} makes \(B\) jointly real
analytic, and the choice of \(U\) makes it zero-free. Finally,
\eqref{eq:return} is an almost-everywhere identity on \(\T^2\). Fubini's
theorem therefore supplies a conull set of parameters \(s\) for which it holds
for almost every \(x\). Intersecting with \(S\cap U\) gives the last
assertion.
\end{proof}

Fix an open interval \(I\subset\R\) on which
\(s\mapsto e^{2\pi i s}\) is a homeomorphism from \(I\) onto \(U\), and
use this lift throughout the rest of the proof.

We now remove the quasiperiodicity imposed by the Zak sewing relation. For
fixed \(s\), write \(F_s(x)=F(x,s)\). A \emph{scalar gauge transformation}
multiplies the unknown fibre by a nowhere-vanishing scalar function:
\(K_s=q_sF_s\). This invertible change preserves nonzero measurable solutions
and transforms
\[
  F_s(x-\theta)=B_s(x)F_s(x)
\]
into
\[
  K_s(x-\theta)
  =\frac{q_s(x-\theta)}{q_s(x)}B_s(x)K_s(x).
\]
Here the canonical choice \(q_s(x)=e^{-2\pi i s x}\) removes exactly the Zak
sewing phase. Thus set
\begin{equation}
  K_s(x)=e^{-2\pi i s x}F(x,s).
  \label{eq:periodic-gauge}
\end{equation}
Then \(K_s(x+1)=K_s(x)\), and \eqref{eq:return} becomes
\begin{equation}
  K_s(x-\theta)=A_s(x)K_s(x),
  \qquad
  A_s(x)=e^{2\pi i s\theta}B_s(x).
  \label{eq:gauged-cocycle}
\end{equation}
We call this the \emph{periodicizing gauge}. Since \(A_s/B_s\) is independent
of \(x\), it leaves the winding number in the \(x\)-variable unchanged.
Moreover, \(|q_s|=1\), so it preserves \(L^2\)-norms and nontriviality.

\section{Irrational returns and the measurable winding obstruction}
\label{sec:winding}

\subsection{Irrational rotations and quantitative returns}

The periodicized fibre equation \eqref{eq:gauged-cocycle} has now reduced the
remaining dynamics to an irrational rotation of the \(x\)-circle. Two
classical facts enter, for distinct purposes: ergodicity controls the zero set
of a measurable solution, while continued-fraction return times control the
phase accumulated by the multiplier.

For \(\theta\notin\Q\), let
\[
  R_\theta x=x-\theta\pmod 1
\]
act on \(\T\). The rotation \(R_\theta\) is ergodic \cite{Petersen}. If
\(p_n/q_n\) are the continued-fraction convergents of \(\theta\), then
\begin{equation}
  q_n\theta-p_n\longrightarrow0.
  \label{eq:cf-return}
\end{equation}
Thus \(R_\theta^{q_n}\) approaches the identity, and the denominators
\(q_n\) provide the natural times at which an iterated cocycle can be
compared with its initial value.

At these return times, accumulated oscillation is controlled by the
Denjoy--Koksma inequality: for every real-valued periodic function \(g\) of
bounded variation,
\begin{equation}
  \sup_{x\in\T}
  \left|
    \sum_{j=0}^{q_n-1}g(x-j\theta)
    -q_n\int_\T g
  \right|
  \leq \Var(g).
  \label{eq:denjoy-koksma}
\end{equation}
For complex-valued \(g\), we apply the estimate to its real and imaginary
parts. In the winding argument below, it controls the derivative of the
accumulated phase; later, it controls the oscillatory part of a periodic
logarithm. Together with \eqref{eq:cf-return}, it converts recurrence of the
base rotation into rigidity of the cocycle. See
\cite{Khinchin,KuipersNiederreiter,KatokHasselblatt}.

\subsection{The winding obstruction}

We now isolate the first dynamical obstruction. Its significance is that it
requires no continuity of the transfer function \(h\); only the multiplier is
smooth.

\begin{lemma}[Measurable winding obstruction]
\label{lem:winding}
Let \(\theta\notin\Q\), let \(0\neq h\in L^2(\T)\), and let
\(b\in C^2(\T;\C^\times)\). If
\begin{equation}
  h(x-\theta)=b(x)h(x)
  \label{eq:abstract-cocycle}
\end{equation}
almost everywhere, then \(\wind(b)=0\).
\end{lemma}

\begin{proof}
\textbf{Step 1: Remove the zero set of the transfer function.}
Because \(b\) never vanishes, the zero set of \(h\) is invariant, modulo null
sets, under the irrational rotation. Ergodicity implies that this zero set has
measure zero or one. Since \(h\neq0\), it has measure zero. Therefore
\[
  u(x)=\frac{h(x)}{|h(x)|}
\]
is defined almost everywhere and has modulus one.

\smallskip
\noindent
\textbf{Step 2: Separate the winding from the periodic phase.}
Let \(\psi=b/|b|\). Then
\[
  u(x-\theta)=\psi(x)u(x)
\]
almost everywhere and \(\wind(\psi)=\wind(b)\). If this common winding is
\(d\in\Z\), there is a real-valued periodic \(C^2\) function \(g\) such that
\[
  \psi(x)=e^{2\pi i(dx+g(x))}.
\]

\smallskip
\noindent
\textbf{Step 3: Iterate through continued-fraction returns.}
Iteration through \(q_n\) steps gives
\[
  u(x-q_n\theta)=e^{2\pi i\Phi_n(x)}u(x),
\]
where
\begin{equation}
  \Phi_n(x)
  =dq_nx-d\theta\frac{q_n(q_n-1)}2
   +\sum_{j=0}^{q_n-1}g(x-j\theta).
  \label{eq:Phi-def}
\end{equation}
Since \(|u|=1\),
\[
  |e^{2\pi i\Phi_n(x)}-1|
  =|u(x-q_n\theta)-u(x)|.
\]
By \eqref{eq:cf-return} and translation continuity in \(L^2(\T)\),
\begin{equation}
  e^{2\pi i\Phi_n}\longrightarrow1
  \quad\text{in }L^2(\T).
  \label{eq:phase-to-one}
\end{equation}
In particular,
\begin{equation}
  \int_\T e^{2\pi i\Phi_n(x)}\dd x\longrightarrow1.
  \label{eq:phase-integral-one}
\end{equation}

\smallskip
\noindent
\textbf{Step 4: Exclude nonzero winding by oscillation.}
Assume \(d\neq0\). Differentiating \eqref{eq:Phi-def} gives
\[
  \Phi_n'(x)
  =dq_n+\sum_{j=0}^{q_n-1}g'(x-j\theta).
\]
The function \(g'\) has mean zero and bounded variation. Denjoy--Koksma
therefore yields
\[
  \sup_x\left|
    \sum_{j=0}^{q_n-1}g'(x-j\theta)
  \right|
  \leq\Var(g').
\]
Hence \(\inf_x|\Phi_n'(x)|\geq c q_n\) for all large \(n\), with \(c>0\).
Moreover,
\[
  \|\Phi_n''\|_{L^1}
  \leq q_n\|g''\|_{L^1}.
\]

Although \(\Phi_n\) need not be periodic, it satisfies
\[
  \Phi_n(x+1)=\Phi_n(x)+dq_n,
\]
and \(\Phi_n'\) is periodic. Consequently, the boundary term in integration
by parts vanishes:
\[
  \left[
    \frac{e^{2\pi i\Phi_n(x)}}{2\pi i\Phi_n'(x)}
  \right]_{0}^{1}=0.
\]
It follows that
\[
  \left|
    \int_\T e^{2\pi i\Phi_n(x)}\dd x
  \right|
  \leq
  \frac{\|\Phi_n''\|_{L^1}}
       {2\pi(\inf|\Phi_n'|)^2}
  =O(q_n^{-1}).
\]
This contradicts \eqref{eq:phase-integral-one}. Therefore \(d=0\).
\end{proof}

\begin{proposition}[Zero winding on the active arc]
\label{prop:zero-winding}
For every \(s\in U\),
\[
  \wind(A_s)=\wind(B_s)=0.
\]
\end{proposition}

\begin{proof}
For almost every \(s\in S\cap U\), Lemma~\ref{lem:active-arc} and
\eqref{eq:gauged-cocycle} give a nonzero periodic fibre
\(K_s\in L^2(\T)\) satisfying the hypotheses of Lemma~\ref{lem:winding}.
Thus \(\wind(A_s)=0\) on a positive-measure subset of \(U\).

The family \(s\mapsto A_s\) is continuous and zero-free on
\(U\times\T\). Winding is integer-valued and locally constant under such a
homotopy. Since \(U\) is connected, the winding vanishes for every
\(s\in U\). Finally, \(A_s=e^{2\pi i s\theta}B_s\) differs from \(B_s\) by
a nonzero factor independent of \(x\), so their windings agree.
\end{proof}

\section{Holonomy quantization and analytic rigidity}
\label{sec:quantization}

Zero winding makes holonomy well defined. The next lemma supplies the exact
continued-fraction estimate needed to understand it. The conclusion is
stronger than the boundedness given directly by Denjoy--Koksma.

\begin{lemma}[Uniform centred returns for \(C^1\) functions]
\label{lem:centered-return}
Let \(\theta\notin\Q\), let \(q_n\) be its convergent denominators, and let
\(\phi\in C^1(\T;\C)\) satisfy \(\int_\T\phi=0\). Then
\begin{equation}
  \sup_{x\in\T}
  \left|
    \sum_{j=0}^{q_n-1}\phi(x-j\theta)
  \right|
  \longrightarrow0.
  \label{eq:centered-return-zero}
\end{equation}
\end{lemma}

\begin{proof}
\textbf{Step 1: Approximate in variation.}
Fix \(\varepsilon>0\). Choose a zero-mean trigonometric polynomial \(P\) so
that
\[
  \Var(\operatorname{Re}(\phi-P))
  +\Var(\operatorname{Im}(\phi-P))<\varepsilon.
\]
Such an approximation follows by periodic smoothing or Fej\'er
approximation in \(W^{1,1}(\T)\), followed by subtraction of the mean.

\smallskip
\noindent
\textbf{Step 2: Control the approximation error uniformly.}
Applying Denjoy--Koksma to the real and imaginary parts of \(\phi-P\) gives
\[
  \sup_x
  \left|
    \sum_{j=0}^{q_n-1}(\phi-P)(x-j\theta)
  \right|<\varepsilon
\]
for every \(n\).

\smallskip
\noindent
\textbf{Step 3: Compute each Fourier mode.}
For every nonzero integer \(k\),
\[
  \sum_{j=0}^{q_n-1}e^{2\pi i k(x-j\theta)}
  =e^{2\pi i kx}
   \frac{1-e^{-2\pi i kq_n\theta}}
        {1-e^{-2\pi i k\theta}}.
\]
The denominator is nonzero because \(\theta\) is irrational, while the
numerator tends to zero by \eqref{eq:cf-return}. Thus every nonconstant
Fourier mode has return sum converging uniformly to zero. Since \(P\) has
only finitely many modes and has zero constant coefficient, the same is true
for \(P\). Therefore the limsup in \eqref{eq:centered-return-zero} is at most
\(\varepsilon\). Letting \(\varepsilon\downarrow0\) proves the claim.
\end{proof}

\begin{lemma}[Holonomy quantization]
\label{lem:holonomy-quantization}
Let \(0\neq h\in L^2(\T)\), let \(b\in C^1(\T;\C^\times)\), and suppose
\[
  h(x-\theta)=b(x)h(x)
\]
almost everywhere. If \(\ell\) is a continuous periodic logarithm of \(b\),
then
\begin{equation}
  \exp\left(q_n\int_0^1\ell(x)\dd x\right)
  \longrightarrow1.
  \label{eq:holonomy-quantization}
\end{equation}
\end{lemma}

\begin{proof}
\textbf{Step 1: Prove that the mean logarithmic modulus is zero.}
As in Lemma~\ref{lem:winding}, the function \(h\) is nonzero almost
everywhere. Set \(w=\log|h|\), which is finite and measurable almost
everywhere. The modulus equation gives
\begin{equation}
  w(x-\theta)-w(x)=\log|b(x)|.
  \label{eq:log-coboundary}
\end{equation}
The function \(w\) need not be integrable, so we use truncation rather than
integrating \eqref{eq:log-coboundary} directly. Let
\[
  \tau_N(t)=\max(-N,\min(t,N)).
\]
Rotation invariance of Haar measure gives
\[
  \int_\T\bigl(\tau_N(w(x-\theta))-\tau_N(w(x))\bigr)\dd x=0.
\]
Because \(\tau_N\) is 1-Lipschitz,
\[
  |\tau_N(w(x-\theta))-\tau_N(w(x))|
  \leq|w(x-\theta)-w(x)|=|\log|b(x)||.
\]
The right side is bounded. Dominated convergence in \(N\) yields
\begin{equation}
  \int_\T\log|b(x)|\dd x=0.
  \label{eq:mean-log-modulus-zero}
\end{equation}

\smallskip
\noindent
\textbf{Step 2: Separate the mean of the logarithm.}
Put
\[
  c=\int_\T\ell(x)\dd x,
  \qquad
  \phi=\ell-c.
\]
Since \(b\in C^1\) and \(e^\ell=b\), differentiation gives
\(\ell'=b'/b\); thus \(\ell\), and hence \(\phi\), belongs to \(C^1(\T)\).
Then \(\int\phi=0\), and \eqref{eq:mean-log-modulus-zero} implies
\(\operatorname{Re}c=0\). Iteration gives
\begin{equation}
  h(x-q_n\theta)
  =e^{q_nc}
   \exp\left(
     \sum_{j=0}^{q_n-1}\phi(x-j\theta)
   \right)h(x).
  \label{eq:holonomy-iterate}
\end{equation}

\smallskip
\noindent
\textbf{Step 3: Pass to the continued-fraction returns.}
Lemma~\ref{lem:centered-return} shows that the exponential factor involving
\(\phi\) converges uniformly to one. Also,
\(h(\,\cdot-q_n\theta)\to h\) in \(L^2(\T)\) by translation continuity.
Equation \eqref{eq:holonomy-iterate} therefore implies
\[
  \|(e^{q_nc}-1)h\|_2\longrightarrow0.
\]
Since \(h\neq0\), this is equivalent to \(e^{q_nc}\to1\), proving
\eqref{eq:holonomy-quantization}.
\end{proof}

Define the resonance group
\begin{equation}
  G_\theta
  =\{\zeta\in S^1:\zeta^{q_n}\to1\}.
  \label{eq:resonance-group}
\end{equation}

\begin{lemma}[The resonance group is Haar-null]
\label{lem:resonance-null}
The set \(G_\theta\) is a Borel subgroup of \(S^1\) and has Haar measure
zero.
\end{lemma}

\begin{proof}
The subgroup property follows directly from the definition, and the
convergence condition expresses \(G_\theta\) as a countable combination of
open conditions, so it is Borel. If it had positive Haar measure, the
Steinhaus theorem for locally compact groups would make it open. Since
\(S^1\) is connected, an open subgroup must be all of \(S^1\).

If \(G_\theta=S^1\), then \(\zeta^{q_n}\to1\) for every \(\zeta\in S^1\).
Dominated convergence would give
\[
  \int_{S^1}\zeta^{q_n}\dd\zeta\longrightarrow1.
\]
The integral on the left is zero for every \(n\), a contradiction.
\end{proof}

\begin{lemma}[Analytic rigidity into the resonance group]
\label{lem:analytic-rigidity}
Let \(I\subset\R\) be a nonempty open interval and let
\(H:I\to\C^\times\) be real analytic. If
\[
  \bigl|\{s\in I:H(s)\in G_\theta\}\bigr|>0,
\]
then \(H\) is constant.
\end{lemma}

\begin{proof}
\textbf{Step 1: Force the image onto the unit circle.}
On a set of positive measure, \(H(s)\in G_\theta\subset S^1\). Hence the
real-analytic function \(|H|^2-1\) vanishes on a set of positive measure and
must vanish identically. Thus \(H(I)\subset S^1\).

\smallskip
\noindent
\textbf{Step 2: Exclude a nonconstant analytic phase.}
Because \(I\) is an interval, \(H\) has a real-analytic lift
\[
  H(s)=e^{i\vartheta(s)}.
\]
If \(H\) were nonconstant, then \(\vartheta'\) would be a nonzero
real-analytic function. Its zero set is discrete. The complement is a
countable union of intervals on which \(\vartheta\) is a local
\(C^1\)-diffeomorphism. On every compact subinterval of such an interval, the
inverse is Lipschitz. Since \(G_\theta\) is Haar-null, its angular preimage is
Lebesgue-null, and therefore so is the preimage under \(\vartheta\). Adding
the discrete critical set still gives a null set, contradicting the
positive-measure hypothesis. Hence \(H\) is constant. See
\cite{KrantzParks} for the real-analytic facts used here.
\end{proof}

\section{Analytic and Laurent holonomy}
\label{sec:laurent}

We now combine the measurable cocycle restrictions with the algebraic form of
the return multiplier. We retain the fixed interval
\(I\subset\R\) lifting the active arc \(U\subset\T\) chosen above. All
functions of \(s\) are understood on this lift.

\begin{lemma}[Joint analytic logarithm]
\label{lem:joint-log}
There is a jointly real-analytic function
\[
  \ell:I\times\T\longrightarrow\C
\]
that is periodic in \(x\) and satisfies
\begin{equation}
  A_s(x)=e^{\ell(s,x)}.
  \label{eq:joint-log}
\end{equation}
\end{lemma}

\begin{proof}
The function \(A:I\times\T\to\C^\times\) is jointly real analytic and
zero-free. Lift it first to the simply connected covering strip
\(I\times\R\). Since the exponential map
\(\exp:\C\to\C^\times\) is a covering map, \(A\) has a continuous logarithm
there. Local analytic inverses of the exponential show that this lift is real
analytic.

The only possible monodromy on the cylinder \(I\times\T\) occurs when the
\(x\)-variable makes one full circuit. The corresponding increment of the
logarithm is
\[
  2\pi i\wind(A_s).
\]
Proposition~\ref{prop:zero-winding} makes this increment zero for every
\(s\in I\). The lifted logarithm is therefore 1-periodic in \(x\) and descends
to \(I\times\T\).
\end{proof}

Define the \emph{periodic holonomy}, namely the holonomy of the cocycle
\(A_s\) obtained after the periodicizing gauge,
\begin{equation}
  H(s)
  =\Hol(A_s)
  =\exp\left(\int_0^1\ell(s,x)\dd x\right),
  \qquad s\in I.
  \label{eq:H-def}
\end{equation}
The function \(H\) is nonzero and real analytic. For almost every
\(s\in S\cap U\), the periodic fibre \(K_s\) is nonzero and satisfies
\eqref{eq:gauged-cocycle}. Lemma~\ref{lem:holonomy-quantization} therefore
gives
\[
  H(s)^{q_n}\longrightarrow1,
\]
so \(H(s)\in G_\theta\) on a positive-measure subset of \(I\). By
Lemma~\ref{lem:analytic-rigidity}, there is a constant \(C\in G_\theta\) such
that
\begin{equation}
  H(s)=C
  \qquad (s\in I).
  \label{eq:H-constant}
\end{equation}

Since
\[
  A_s(x)=e^{2\pi i s\theta}B_s(x),
\]
the function \(\ell(s,x)-2\pi i s\theta\) is a periodic logarithm of
\(B_s(x)\). Let
\[
  W=\{e^{2\pi i s}:s\in I\}\subset S^1
\]
and define the Laurent holonomy
\begin{equation}
  J(e^{2\pi i s})
  =\exp\left(
    \int_0^1\bigl(\ell(s,x)-2\pi i s\theta\bigr)\dd x
  \right).
  \label{eq:J-def}
\end{equation}
Equations \eqref{eq:H-constant} and \eqref{eq:J-def} give the
irrational-character law
\begin{equation}
  J(e^{2\pi i s})=Ce^{-2\pi i\theta s}
  \qquad (s\in I).
  \label{eq:irrational-character-law}
\end{equation}

The next proposition explains the algebraic meaning of Laurent holonomy. We
include the full root-product argument because this is the point where the
analytic and algebraic parts of the proof meet.

\begin{proposition}[Local Laurent holonomy certificate]
\label{prop:laurent-holonomy}
Let \(W\subset S^1\) be a nonempty open arc and let
\[
  L(w,z)\in\C[w^{\pm1},z^{\pm1}]
\]
be nonzero. Suppose
\[
  L(w,z)\neq0
  \qquad (w\in W,\ |z|=1)
\]
and suppose that \(z\mapsto L(w,z)\) has winding zero for every \(w\in W\).
Define
\begin{equation}
  J(w)
  =\exp\left(
    \int_0^1\log L(w,e^{2\pi i x})\dd x
  \right),
  \label{eq:abstract-J}
\end{equation}
where the logarithm is continuous and periodic in \(x\). After possibly
shrinking \(W\), there are Laurent polynomials
\[
  A_0,\ldots,A_D\in\C[w^{\pm1}],
\]
not all zero, such that
\begin{equation}
  \sum_{r=0}^{D}A_r(w)J(w)^r=0
  \qquad (w\in W).
  \label{eq:laurent-certificate}
\end{equation}
\end{proposition}

\begin{proof}
\textbf{Step 1: Convert the Laurent polynomial in \(z\) to a polynomial.}
Let \(n_-\) be the least exponent of \(z\) occurring in \(L\), and set
\(N=-n_-\). Then
\[
  P(w,z)=z^NL(w,z)=\sum_{j=0}^{d}a_j(w)z^j,
  \qquad a_j\in\C[w^{\pm1}],
\]
is a polynomial in \(z\), and \(a_0,a_d\) are not identically zero. A
nonzero one-variable Laurent polynomial cannot vanish on an open arc of
\(S^1\). We may therefore shrink \(W\) so that
\[
  a_0(w)a_d(w)\neq0
  \qquad(w\in W).
\]

\smallskip
\noindent
\textbf{Step 2: Count and separate the inside roots.}
For fixed \(w\in W\), the argument principle gives
\[
  \#\{\text{zeros of }P(w,\cdot)\text{ in }|z|<1\}
  =\wind(P(w,\cdot))
  =N+\wind(L(w,\cdot))=N,
\]
with multiplicity. In particular, \(N\geq0\). Because no root lies on the
unit circle, Rouch\'e's theorem and compactness on a smaller closed subarc show
that the \(N\) inside roots remain uniformly separated from the \(d-N\)
outside roots. This is the only root-separation statement needed below.

\smallskip
\noindent
\textbf{Step 3: Compute the exponential mean logarithm.}
In a local algebraic splitting field, write
\[
  P(w,z)=a_d(w)\prod_{j=1}^{d}(z-r_j(w)),
\]
where, counting multiplicity,
\[
  |r_j(w)|<1\quad(1\leq j\leq N),
  \qquad
  |r_j(w)|>1\quad(N<j\leq d).
\]
On \(|z|=1\), the cancellation of the \(N\) inside factors with \(z^{-N}\)
gives
\begin{equation}
  L(w,z)
  =a_d(w)
   \prod_{j=1}^{N}(1-r_j(w)z^{-1})
   \prod_{j=N+1}^{d}(z-r_j(w)).
  \label{eq:inside-outside-factorization}
\end{equation}
For \(|r|<1\), the power-series logarithm of \(1-rz^{-1}\) has zero constant
Fourier coefficient. For \(|r|>1\), write
\[
  z-r=-r(1-z/r);
\]
again the logarithm of the second factor has zero mean. Exponentiating the
mean logarithm in \eqref{eq:inside-outside-factorization} therefore gives the
exact identity
\begin{equation}
  J(w)
  =(-1)^{d-N}a_d(w)
   \prod_{j=N+1}^{d}r_j(w).
  \label{eq:exterior-root-product}
\end{equation}
This is the complex mean-logarithm form of Jensen's formula; see
\cite{Ahlfors,Conway}.

\smallskip
\noindent
\textbf{Step 4: Pass from roots to a Laurent relation.}
Every root \(r_j\) is algebraic over \(\C(w)\), because it satisfies
\(P(w,r_j)=0\). Hence the exterior-root product in
\eqref{eq:exterior-root-product}, and therefore \(J\), belongs to a finite
algebraic extension of \(\C(w)\). It satisfies a nontrivial polynomial
relation
\[
  \sum_{r=0}^{D}R_r(w)J(w)^r=0,
  \qquad R_r\in\C(w).
\]
Clear the common denominator and multiply by a power of \(w\). The resulting
coefficients belong to \(\C[w^{\pm1}]\), are not all zero, and give
\eqref{eq:laurent-certificate}. Multiplicities and local permutations of the
roots cause no difficulty: the separated exterior-root product is a local
algebraic branch, and every such branch satisfies the same type of polynomial
relation over \(\C(w)\).
\end{proof}

The algebraic certificate is incompatible with the character law
\eqref{eq:irrational-character-law}.

\begin{lemma}[No Laurent relation for an irrational character]
\label{lem:no-laurent-relation}
Let \(\theta\notin\Q\) and \(C\in\C^\times\). Suppose
\(A_0,\ldots,A_D\in\C[w^{\pm1}]\) satisfy
\begin{equation}
  \sum_{r=0}^{D}
  A_r(e^{2\pi i s})
  \bigl(Ce^{-2\pi i\theta s}\bigr)^r=0
  \label{eq:putative-laurent-relation}
\end{equation}
on a nonempty real interval. Then every \(A_r\) is zero.
\end{lemma}

\begin{proof}
Expand each Laurent polynomial and combine like terms. The left side of
\eqref{eq:putative-laurent-relation} becomes a finite exponential sum
\[
  \sum_{(k,r)\in\mathcal F}
  b_{k,r}e^{2\pi i(k-r\theta)s}.
\]
If
\[
  k-r\theta=k'-r'\theta,
\]
then \((r-r')\theta=k-k'\). Irrationality implies \(r=r'\) and \(k=k'\).
Thus all frequencies are distinct. A finite exponential sum with distinct
real frequencies cannot vanish on an interval unless every coefficient is
zero. One way to see this is to differentiate at one point through the number of terms and
use the resulting Vandermonde determinant. Since \(C\neq0\), the vanishing of
the coefficients \(b_{k,r}\) is exactly the vanishing of all Laurent
coefficients of all \(A_r\).
\end{proof}

\section{Proof of the main theorem}
\label{sec:main-proof}

\begin{proof}[Proof of Theorem~\ref{thm:main}]
Assume that a nontrivial dependence exists.

\smallskip
\noindent
\textbf{Step 1: Normalize the configuration.}
Proposition~\ref{prop:normal-form} transports the configuration to
\[
  0,\quad \lambda_1,\quad \lambda_2,\quad (a,p/m),
\]
where \(\lambda_1,\lambda_2\) form a unimodular integer basis and
\(a\notin\Q\). The three-point theorem forces every coefficient in the
relation to be nonzero.

\smallskip
\noindent
\textbf{Step 2: Produce the irrational return cocycle.}
The Zak transform gives the one-step equation \eqref{eq:one-step}.
Proposition~\ref{prop:return} closes the rational frequency displacement after
\(m\) steps and yields
\[
  F(x-\theta,s)=B_s(x)F(x,s),
  \qquad \theta=ma\notin\Q,
\]
where
\[
  B_s(x)=L(e^{2\pi i s},e^{2\pi i x})
\]
for a nonzero Laurent polynomial \(L\).

\smallskip
\noindent
\textbf{Step 3: Isolate a positive-measure family of zero-free fibres.}
Lemma~\ref{lem:trinomial-zeros} and Corollary~\ref{cor:exceptional} show that
only finitely many transverse parameters can contain a zero of \(B_s\). By
Lemma~\ref{lem:active-arc}, there is a connected arc \(U\) on which \(B\) is
zero-free and on which a positive-measure family of Zak fibres is nonzero.
After applying the periodicizing gauge \eqref{eq:periodic-gauge}, those
fibres satisfy
\[
  K_s(x-\theta)=A_s(x)K_s(x),
  \qquad A_s(x)=e^{2\pi i s\theta}B_s(x).
\]

\smallskip
\noindent
\textbf{Step 4: Force zero winding and constant periodic holonomy.}
The measurable winding obstruction, Lemma~\ref{lem:winding}, gives zero
winding on the active fibres. Homotopy invariance extends this conclusion to
every \(s\in U\). Hence the family admits the joint periodic logarithm of
Lemma~\ref{lem:joint-log}. Holonomy quantization places the holonomy \(H(s)\)
of the periodic cocycle \(A_s\) in \(G_\theta\) on a positive-measure subset.
Since
\(G_\theta\) is Haar-null, analytic rigidity forces
\[
  H(s)=C
\]
throughout the arc.

\smallskip
\noindent
\textbf{Step 5: Recover the irrational character.}
Undoing the gauge---equivalently, passing from \(A_s\) back to the original
return multiplier \(B_s\)---gives the Laurent holonomy
\[
  J(e^{2\pi i s})=Ce^{-2\pi i\theta s}
\]
on the arc.

\smallskip
\noindent
\textbf{Step 6: Contradict Laurent algebraicity.}
Proposition~\ref{prop:laurent-holonomy} gives a nontrivial Laurent-polynomial
relation for \(J\) on a smaller nonempty arc. Substituting the preceding
irrational-character law produces a Laurent relation of the form excluded by
Lemma~\ref{lem:no-laurent-relation}. This contradiction rules out the assumed
dependence and proves the theorem.
\end{proof}

\section{Worked example: Heil's
\texorpdfstring{\(\sqrt2\)}{sqrt(2)}-configuration}
\label{sec:heil-worked}

We now apply the proof architecture to one configuration from
Conjecture~9.2 of Heil's survey. We choose part~(a), because the same
irrational number governs both coordinates of the rogue point before
normalization. An integral symplectic change of variables turns this visible
rational relation into a closed coordinate and leaves a one-step irrational
return. The resulting calculation exhibits the geometry, the Zak cocycle,
the exceptional fibres, and the Laurent contradiction without suppressing
any intermediate step.

\smallskip
\noindent
\textbf{Step 1: Place Heil's system in phase space.}
In Heil's positive-modulation convention, part~(a) asks whether, for every
nonzero \(g\in L^2(\R)\), the system
\begin{equation}
  \mathcal H_{\sqrt2}(g)
  =\left\{
    g(t),\ g(t-1),\ e^{2\pi it}g(t),\
    e^{2\pi i\sqrt2\,t}g(t-\sqrt2)
  \right\}
  \label{eq:heil-system}
\end{equation}
is linearly independent
\cite[Conjecture~9.2(a), p.~173]{HeilSurvey}. Our convention is
\(M_\omega f(t)=e^{-2\pi i\omega t}f(t)\). Thus the corresponding set of
time-frequency points is
\begin{equation}
  \Lambda_H
  =\left\{
    (0,0),\ (1,0),\ (0,-1),\ (\sqrt2,-\sqrt2)
  \right\}.
  \label{eq:heil-points}
\end{equation}
Take the oriented lattice basis
\[
  u=(0,-1),\qquad v=(1,0).
\]
Then \(\sigma(u,v)=1\), and the rogue point satisfies
\[
  (\sqrt2,-\sqrt2)=\sqrt2\,u+\sqrt2\,v,
  \qquad
  \dim_{\Q}\operatorname{span}_{\Q}\{1,\sqrt2,\sqrt2\}=2.
\]
The configuration is therefore critical and mixed arithmetic.

\smallskip
\noindent
\textbf{Step 2: Compute the integral symplectic normal form.}
Consider
\begin{equation}
  C=
  \begin{pmatrix}
    0&-1\\
    1& 1
  \end{pmatrix}
  \in\SL(2,\Z).
  \label{eq:heil-C}
\end{equation}
Since \(\det C=1\), the matrix is symplectic. Direct multiplication gives
\begin{align}
  C(1,0)&=(0,1)=:\lambda_2,\notag\\
  C(0,-1)&=(1,-1)=:\lambda_1,\notag\\
  C(\sqrt2,-\sqrt2)&=(\sqrt2,0).
  \label{eq:heil-C-action}
\end{align}
The vectors \(\lambda_1=(1,-1)\) and \(\lambda_2=(0,1)\) form a
\(\Z\)-basis of \(\Z^2\), with
\(\sigma(\lambda_1,\lambda_2)=1\). The fourth point is now in the mixed
normal form
\[
  \left(a,\frac{p}{m}\right)=(\sqrt2,0),
  \qquad a=\sqrt2,\quad p=0,\quad m=1.
\]
Metaplectic covariance implements \(C\) by a unitary operator; the projective
phases can be absorbed into the coefficients of a putative dependence.
Figure~\ref{fig:heil-normalization} records this normalization geometrically.

\begin{figure}[t]
\centering
\includegraphics[width=\textwidth]
{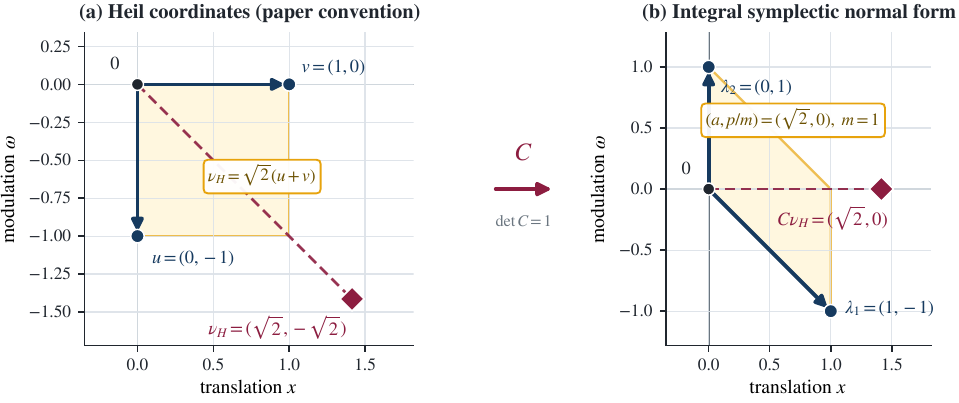}
\caption{The phase-space geometry of Heil's Conjecture~9.2(a), expressed in
the sign convention of this paper. The shaded parallelogram has symplectic
area one. The integral symplectic matrix \(C\) sends the rogue point
\((\sqrt2,-\sqrt2)\) to \((\sqrt2,0)\), so the rational coordinate closes
after one Zak step while the remaining translation is the irrational
rotation by \(\sqrt2\).}
\label{fig:heil-normalization}
\end{figure}

\smallskip
\noindent
\textbf{Step 3: Write the normalized dependence and its Zak equation.}
Suppose that \eqref{eq:heil-system} were dependent. After applying the
metaplectic normalization, there would be a nonzero \(f\in L^2(\R)\) and
coefficients \(c_0,c_1,c_2,c_3\in\C\) such that
\begin{equation}
  c_0f+c_1\pi(1,-1)f+c_2\pi(0,1)f
  +c_3\pi(\sqrt2,0)f=0.
  \label{eq:heil-normalized-dependence}
\end{equation}
The points are distinct, so the three-point theorem forces
\begin{equation}
  c_0c_1c_2c_3\neq0.
  \label{eq:heil-coefficients-nonzero}
\end{equation}
Set \(F=Zf\). For \(\lambda_1=(r_1,t_1)=(1,-1)\) and
\(\lambda_2=(r_2,t_2)=(0,1)\), the lattice trinomial
\eqref{eq:P-def} becomes
\begin{equation}
  P(x,s)
  =c_0+c_1e^{2\pi i(x-s)}+c_2e^{-2\pi ix}.
  \label{eq:heil-P}
\end{equation}
The rogue frequency coordinate is zero. Consequently, the Zak transform of
\eqref{eq:heil-normalized-dependence} is already the exact return equation
\begin{equation}
  F(x-\sqrt2,s)=B_s(x)F(x,s),
  \qquad
  B_s(x)=-c_3^{-1}P(x,s).
  \label{eq:heil-return}
\end{equation}
Writing
\[
  w=e^{2\pi is},\qquad z=e^{2\pi ix},
\]
we obtain the explicit Laurent form
\begin{equation}
  B(w,z)
  =-\frac1{c_3}
    \left(c_0+c_1zw^{-1}+c_2z^{-1}\right)
  =-\frac{c_1z^2+c_0wz+c_2w}{c_3wz}.
  \label{eq:heil-Laurent-return}
\end{equation}
Thus the rational closing parameter is \(m=1\), the irrational return is
\(\theta=\sqrt2\), and the entire return multiplier is a Laurent trinomial.

\smallskip
\noindent
\textbf{Step 4: Isolate the active zero-free arc and periodicize the fibre.}
Because the two nonconstant characters in \eqref{eq:heil-P} form a
unimodular pair, Lemma~\ref{lem:trinomial-zeros} shows that \(P\), and hence
\(B\), has at most two zeros on \(\T^2\). Corollary~\ref{cor:exceptional}
and Lemma~\ref{lem:active-arc} therefore give a connected arc \(U\) on which
\(B_s(x)\neq0\) for every \(x\), while a positive-measure family of Zak
fibres remains nonzero. The canonical gauge
\begin{equation}
  K_s(x)=e^{-2\pi isx}F(x,s)
  \label{eq:heil-periodic-gauge}
\end{equation}
makes the fibre periodic and transforms \eqref{eq:heil-return} into
\begin{equation}
  K_s(x-\sqrt2)=A_s(x)K_s(x),
  \qquad
  A_s(x)=e^{2\pi i\sqrt2s}B_s(x).
  \label{eq:heil-periodic-cocycle}
\end{equation}
The measurable winding obstruction gives \(\wind(B_s)=\wind(A_s)=0\) on
the active arc. Holonomy quantization and analytic rigidity then force the
periodic holonomy of \(A_s\) to equal a constant \(C\in G_{\sqrt2}\).
Undoing the gauge yields
\begin{equation}
  J(e^{2\pi is})=Ce^{-2\pi i\sqrt2s}
  \label{eq:heil-character-law}
\end{equation}
on a nonempty real interval lifting a subarc of \(U\).

\smallskip
\noindent
\textbf{Step 5: Compute the Laurent certificate explicitly.}
For fixed \(w\) in that subarc, set
\begin{equation}
  Q_w(z)=c_1z^2+c_0wz+c_2w.
  \label{eq:heil-Q}
\end{equation}
Equation~\eqref{eq:heil-Laurent-return} shows that
\[
  zB(w,z)=-\frac{1}{c_3w}Q_w(z).
\]
Since \(B(w,\cdot)\) has winding zero, the polynomial
\(zB(w,\cdot)\) has winding one. The argument principle therefore places
exactly one root of \(Q_w\) inside the unit circle and the other outside.
Denote them by \(r_{\mathrm{in}}(w)\) and
\(r_{\mathrm{out}}(w)\), respectively. Factoring gives
\begin{align*}
  B(w,z)
  &=-\frac{c_1}{c_3w}z^{-1}
    \bigl(z-r_{\mathrm{in}}(w)\bigr)
    \bigl(z-r_{\mathrm{out}}(w)\bigr)\\
  &=-\frac{c_1}{c_3w}
    \bigl(1-r_{\mathrm{in}}(w)z^{-1}\bigr)
    \bigl(z-r_{\mathrm{out}}(w)\bigr).
\end{align*}
The first parenthesis has mean logarithm zero. Since
\(|r_{\mathrm{out}}(w)|>1\), write
\[
  z-r_{\mathrm{out}}(w)
  =-r_{\mathrm{out}}(w)
   \left(1-\frac{z}{r_{\mathrm{out}}(w)}\right);
\]
the last parenthesis also has mean logarithm zero. Hence the Laurent holonomy
is
\begin{equation}
  J(w)=\frac{c_1}{c_3w}r_{\mathrm{out}}(w).
  \label{eq:heil-J-root}
\end{equation}
Substituting
\(r_{\mathrm{out}}(w)=c_3wJ(w)/c_1\) into
\(Q_w(r_{\mathrm{out}}(w))=0\) gives the explicit quadratic certificate
\begin{equation}
  c_3^2wJ(w)^2+c_0c_3wJ(w)+c_1c_2=0.
  \label{eq:heil-quadratic-certificate}
\end{equation}
This is the local Laurent-algebra conclusion of
Proposition~\ref{prop:laurent-holonomy}, specialized to Heil's configuration.

\smallskip
\noindent
\textbf{Step 6: Read the contradiction directly.}
Insert \eqref{eq:heil-character-law} and \(w=e^{2\pi is}\) into
\eqref{eq:heil-quadratic-certificate}. We obtain
\begin{equation}
  c_3^2C^2e^{2\pi i(1-2\sqrt2)s}
  +c_0c_3Ce^{2\pi i(1-\sqrt2)s}
  +c_1c_2=0
  \label{eq:heil-three-frequencies}
\end{equation}
on a nonempty interval. The three frequencies
\[
  1-2\sqrt2,\qquad 1-\sqrt2,\qquad 0
\]
are distinct, and every coefficient in
\eqref{eq:heil-three-frequencies} is nonzero by
\eqref{eq:heil-coefficients-nonzero} and \(C\neq0\). A finite exponential
sum with distinct frequencies cannot vanish on an interval. This
contradiction proves that the system \(\mathcal H_{\sqrt2}(g)\) in
\eqref{eq:heil-system} is linearly independent for every nonzero
\(g\in L^2(\R)\).

\smallskip
\noindent
\textbf{Step 7: Visualize the zero set and the root separation.}
For the visualization only, take
\(c_0=c_1=c_2=c_3=1\). No nonzero Zak solution is being asserted. In this
coefficient slice,
\begin{align}
  B_s(x)
  &=-\left(1+e^{2\pi i(x-s)}+e^{-2\pi ix}\right)\notag\\
  &=-\left(1+2e^{-\pi is}\cos(2\pi x-\pi s)\right).
  \label{eq:heil-equal-B}
\end{align}
If \(c=\cos(2\pi x-\pi s)\), then
\begin{equation}
  |B_s(x)|^2=1+4c^2+4c\cos(\pi s).
  \label{eq:heil-B-square}
\end{equation}
The minimum over \(c\in[-1,1]\) occurs at
\(c=-\tfrac12\cos(\pi s)\), and therefore
\begin{equation}
  \min_{x\in\T}|B_s(x)|=|\sin(\pi s)|.
  \label{eq:heil-exact-margin}
\end{equation}
In the fundamental square \([0,1)_x\times[0,1)_s\), the only zeros are
\begin{equation}
  (x,s)=\left(\frac13,0\right),
  \qquad
  (x,s)=\left(\frac23,0\right).
  \label{eq:heil-exact-zeros}
\end{equation}
Thus \(s=0\) is the sole exceptional fibre. In the same slice,
\(Q_w(z)=z^2+wz+w\) and
\eqref{eq:heil-quadratic-certificate} reduces to
\begin{equation}
  wJ(w)^2+wJ(w)+1=0.
  \label{eq:heil-equal-certificate}
\end{equation}
Figures~\ref{fig:heil-return} and \ref{fig:heil-root-holonomy} display these
three aspects of the same calculation at publication scale.

\begin{figure}[t]
\centering
\includegraphics[width=\textwidth]
{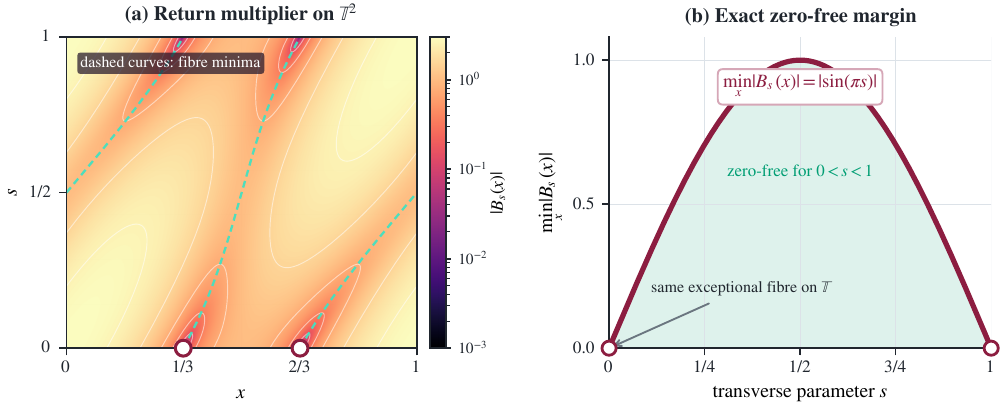}
\caption{The return multiplier in the normalized equal-coefficient slice of
Heil's Conjecture~9.2(a). Panel~(a) shows \(|B_s(x)|\) on a logarithmic colour
scale. The dashed curves are the two exact minimizing branches, and the marked
points \((1/3,0)\) and \((2/3,0)\) are the only zeros in the fundamental
square. Panel~(b) displays the exact fibrewise margin
\(\min_x|B_s(x)|=|\sin(\pi s)|\). It is strictly positive for \(0<s<1\), while
the endpoints represent the same exceptional fibre of \(\T\).}
\label{fig:heil-return}
\end{figure}

\begin{figure}[t]
\centering
\includegraphics[width=\textwidth]
{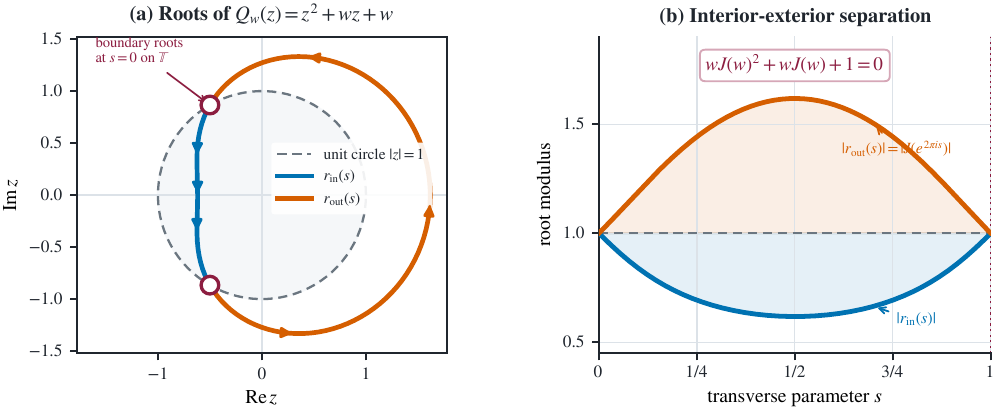}
\caption{The root geometry underlying the holonomy calculation. Panel~(a)
follows the two roots of \(Q_w(z)=z^2+wz+w\) as \(w=e^{2\pi is}\) traverses the
unit circle. Away from the exceptional parameter, one root remains strictly
inside and the other strictly outside \(|z|=1\); the arrows record their
orientations. Panel~(b) isolates the corresponding modulus separation. The
exterior branch defines the holonomy \(J\) and satisfies the exact algebraic
certificate \(wJ(w)^2+wJ(w)+1=0\) from
\eqref{eq:heil-equal-certificate}.}
\label{fig:heil-root-holonomy}
\end{figure}

All three figures are generated reproducibly by
\nolinkurl{anc/generate_heil_conjecture_9_2_figures.py}. The script checks
the symplectic transformation, the two exact zeros, the fibrewise minimum,
the inside--outside root count, and the quadratic holonomy relation before
exporting the vector graphics.

\clearpage
\appendix

\section{Scope and reproducibility of the Lean 4 certification}
\label{app:formalization}

The accompanying supplementary archive \certarchive{} supplies an end-to-end
Lean~4 certification of the formal statement corresponding to
Theorem~\ref{thm:main}. The headline endpoint is
\nolinkurl{CriticalMixedHRT.critical_mixed_orbit_HRT_L2} in
\nolinkurl{RequestProject/MainStatement.lean}. With the convention
\[
  \pi(x,\omega)f(t)=e^{-2\pi i\omega t}f(t-x),
\]
the theorem assumes exactly the critical symplectic condition, distinctness
of \(0,u,v,\nu\), the identity \(\nu=\alpha u+\beta v\), rational rank two,
and a nonzero \(L^2\) window. It concludes that every coefficient in an
almost-everywhere four-term relation is zero.

The certification is not conditional on a normal-form interface. In
particular, the former assumption
\nolinkurl{StandardCriticalMixedNormalForm} is absent from the dependency chain.
The physical all-nonzero relation is ruled out by
\nolinkurl{physical_critical_mixed_all_nonzero_relation_impossible}; the
remaining coefficient cases are discharged through the formally proved
three-point theorem. The final status is summarized in
Table~\ref{tab:lean-status}.

The three-point theorem is itself exported in two unconditional forms:
\nolinkurl{CriticalMixedHRT.three_point_HRT_L2} gives the physical
almost-everywhere coefficient statement for every injective triple, while
\nolinkurl{CriticalMixedHRT.three_point_HRT_L2_linearIndependent} gives literal
complex linear independence of the corresponding vectors in the canonical
\(L^2(\R)\) quotient. Thus the three-point result is a checked theorem in the
dependency chain, not an imported or assumed interface.

\begin{table}[htbp]
\centering
\caption{Certification status of the load-bearing Lean endpoints. Colour is
paired with an explicit textual status so that no information is conveyed by
colour alone.}
\label{tab:lean-status}
\small
\renewcommand{\arraystretch}{1.18}
\begin{tabularx}{0.98\textwidth}{
  >{\raggedright\arraybackslash}p{0.25\textwidth}
  >{\centering\arraybackslash}p{0.18\textwidth}
  >{\raggedright\arraybackslash}X}
\toprule
\rowcolor{HRTHeader}
\textbf{Component} & \textbf{Status} & \textbf{Audited result}\\
\midrule
Physical all-nonzero case
& \cellcolor{HRTPositive}\textbf{Complete}
& A contradiction is proved from the manuscript hypotheses; no assumed
normal-form endpoint is used.\\
Complete three-point HRT theorem
& \cellcolor{HRTPositive}\textbf{Complete}
& Both the physical almost-everywhere formulation and literal complex linear
independence in \(L^2(\R)\) are proved for every injective triple.\\
Three-point coefficient reduction
& \cellcolor{HRTPositive}\textbf{Complete}
& The formally proved three-point theorem shows that a nontrivial four-term
relation would have all four coefficients nonzero.\\
Main theorem
& \cellcolor{HRTPositive}\textbf{Complete}
& \nolinkurl{critical_mixed_orbit_HRT_L2} has a closed proof term and builds
as part of \nolinkurl{RequestProject.Main}.\\
Source hygiene
& \cellcolor{HRTPositive}\textbf{Clean}
& No executable \texttt{sorry}, \texttt{admit}, project \texttt{axiom},
\texttt{unsafe} declaration, or \nolinkurl{implemented_by} escape hatch occurs
in the audited source.\\
Kernel axiom closure
& \cellcolor{HRTPresent}\textbf{Foundational only}
& Lean reports exactly
\texttt{propext}, \texttt{Classical.choice}, and \texttt{Quot.sound};
\texttt{sorryAx} is absent.\\
Exposition, bibliography, and figure
& \cellcolor{HRTPartial}\textbf{Human audited}
& These presentation-level components are deliberately outside the scope of
kernel certification and have been checked separately.\\
\bottomrule
\end{tabularx}
\end{table}

The project is pinned to Lean \texttt{4.28.0} and Mathlib \texttt{v4.28.0},
revision \nolinkurl{8f9d9cff6bd728b17a24e163c9402775d9e6a365}. From the archive root, the
reproducibility commands are
\begin{verbatim}
lake exe cache get
lake build RequestProject.Main
lake env lean RequestProject/ThreePointCertificationAxioms.lean
lake env lean RequestProject/CertificationAxioms.lean
\end{verbatim}
The audited integrated build completed successfully across 8,164 jobs. The full packaged
archive has SHA-256 checksum
\begin{center}
\small\nolinkurl{f12716da6e978ce8054ecbddff1051d79d9f923e89f3fa7c3d3734e4778a62c4}.
\end{center}
This checksum, the toolchain files, the complete axiom report, and the source
hygiene commands are included in the companion. Thus ``end-to-end'' here has
a precise meaning: Lean's kernel checks the formal main theorem from its
stated hypotheses using only the three ordinary foundational principles just
listed. It does not mean that the kernel certifies the historical narrative,
bibliographic metadata, or the natural-language correspondence between every
sentence of this manuscript and the formal source. For background on Lean and
Mathlib, see \cite{deMouraUllrich,Mathlib}.

\section*{Acknowledgments}

The author is grateful to Christopher Heil for the clarity and persistence
with which he has sustained the HRT problem; to Kasso Okoudjou for his
mentorship, advocacy, and mathematical insight; and to Bradley Currey and
Darrin Speegle for the conversations through which the problem first became
part of the author's mathematical life. The author also thanks Akram Aldroubi
and the organizers of ICCHA 2026 for the invitation to participate in the
conference and to contribute to this special issue.

\section*{Statements and declarations}

\noindent
\textbf{Funding.}
This material is based upon work supported by the National Science Foundation
under Grant DMS-2205852, \emph{Collaborative Research: Topics in Abstract,
Applied, and Computational Harmonic Analysis}. Any opinions, findings, and
conclusions or recommendations expressed in this material are those of the
author and do not necessarily reflect the views of the National Science
Foundation.

\smallskip
\noindent
\textbf{Competing interests.}
The author declares that he has no competing financial or non-financial
interests that are directly or indirectly related to this work.

\smallskip
\noindent
\textbf{Data availability.}
No datasets were generated or analyzed in the preparation of this article.

\smallskip
\noindent
\textbf{Code availability.}
The complete Lean~4 certification, its pinned toolchain, build instructions,
axiom audit, and source-hygiene report are supplied as the supplementary
archive \certarchive{}. The reproducible Python source for
Figures~\ref{fig:heil-normalization}, \ref{fig:heil-return}, and
\ref{fig:heil-root-holonomy}, including its internal algebraic and numerical
checks, is included in the manuscript source package.

\smallskip
\noindent
\textbf{Use of generative artificial intelligence.}
This manuscript was prepared with assistance from ChatGPT (OpenAI). ChatGPT
assisted with LaTeX reconstruction, mathematical exploration, adversarial
proof auditing, and the drafting and revision of the exposition. No
language-model output is used as a mathematical premise: the claims rest on
the arguments presented in the paper. ChatGPT is not an author. The human
author reviewed and revised the manuscript and remains responsible for every
mathematical claim and for the final submitted text.

\smallskip
\noindent
\textbf{Author contributions.}
The author is solely responsible for the conceptualization, mathematical
arguments, writing, and final verification of the manuscript.

\end{document}